\documentclass[11pt]{amsart}

\usepackage{amssymb,amsmath,amsfonts,amsthm}
\usepackage{mathrsfs}
\usepackage[small,nohug,heads=vee]{diagrams}
\diagramstyle[labelstyle=\scriptstyle]

\theoremstyle{plain}
\newtheorem{thm}{Theorem}[section]
\newtheorem{lem}[thm]{Lemma}

\newtheorem{cor}[thm]{Corollary}
\newtheorem{claim}{Claim}

\theoremstyle{definition}
\newtheorem{defn}{Definition}[section]

\newtheorem{exmp}{Example}[section]

\theoremstyle{remark}
\newtheorem{rem}{Remark}[section]

\title{On a generalization of a theorem of McDuff}
\author{G.~Deltour}
\address{D\'epartement de Math\'ematiques, CC051\\
Universit\'e Montpellier II\\
Place Eug\`ene Bataillon\\
F-34095 Montpellier Cedex, France\\}
\email{guillaume.deltour@math.univ-montp2.fr}
\date{\today}

\newcommand{\R}{\mathbb{R}}
\newcommand{\C}{\mathbb{C}}

\newcommand{\got}[1]{\mathfrak{#1}}
\newcommand{\Ad}{\mathop{\mathrm{Ad}}}
\newcommand{\ad}{\mathop{\mathrm{ad}}}

\newcommand{\Orb}{\mathcal{O}}

\newcommand{\Chol}{\mathcal{C}_{\mathrm{hol}}}

\newcommand{\id}{\mathrm{id}}

\begin{document}

\begin{abstract}
We study the symplectic structure of the holomorphic coadjoint orbits, generalizing a theorem of McDuff on the symplectic structure of Hermitian symmetric spaces of noncompact type.
\end{abstract}

\maketitle


\section{Introduction}

In this article, we are interested in the symplectic structure of holomorphic coadjoint orbits.

Let G be a noncompact, connected, real semisimple Lie group with finite center, and let $K$ be a maximal compact subgroup of $G$. We denote by $\got{g}$ and $\got{k}$ the Lie algebras of $G$ and $K$ respectively. The maximal compact subgroup $K$ is connected and corresponds to a Cartan decomposition $\got{g}=\got{k}\oplus\got{p}$ on the Lie algebra level, see for instance \cite{knapp}. Assume that $G/K$ is a Hermitian symmetric space, that is, there exists a $G$-invariant complex structure on the manifold $G/K$. Let $T$ be a maximal torus in $K$ and $\got{t}$ its Lie algebra. We fix a Weyl chamber $\got{t}_+^*\subset\got{t}^*$ for the compact group $K$. We recall that a coadjoint orbit $\Orb$ of the noncompact group $G$ is said \emph{elliptic} if $\Orb$ intersects $\got{k}^*$.

Any coadjoint orbit $\Orb$ carries a canonical $G$-invariant symplectic form $\Omega_{\Orb}$, called the Kirillov-Kostant-Souriau symplectic form on $\Orb$. If $\Orb=G\cdot\lambda$ is the coadjoint orbit through $\lambda\in\got{g}^*$, the symplectic form $\Omega_{G\cdot\lambda}$ is defined above $\lambda$ by
\[
(\Omega_{G\cdot\lambda})|_{\lambda}(X,Y) := \langle\lambda,[X,Y]\rangle, \text{ for all $X,Y\in\got{g}/\got{g}_{\lambda}$,}
\]
where we identify canonically the tangent space $T_{\lambda}(G\cdot\lambda)$ with the vector space $\got{g}/\got{g}_{\lambda}$. For instance, considering the homogeneous space $G/K$ as the coadjoint orbit of some generic element of the center of $\got{k}^*$, the Kirillov-Kostant-Souriau form defines a natural $G$-invariant symplectic structure on the Hermitian symmetric space $G/K$, which is compatible with the $G$-invariant complex structure on $G/K$.

Actually, the Hermitian symmetric spaces form a part of a much larger family of K\"ahler coadjoint orbits, called \emph{holomorphic coadjoint orbits}. They are the coadjoint orbits which are related to the holomorphic discrete series of $G$.

More precisely, a holomorphic coadjoint orbit is an elliptic coadjoint orbit $\Orb$ of $G$ which carries a canonical $G$-invariant K\"ahler structure, compatible with the Kirillov-Kostant-Souriau symplectic structure on $\Orb$. The holomorphic coadjoint orbits are parametrized by a subchamber $\Chol$ of $\got{t}_+^*$, called the \emph{holomorphic chamber}. This chamber can be formally defined using noncompact roots of $G$, see paragraph \ref{subsection:holomorphic_chamber}. This definition has the following consequence: if $\lambda\in\Chol$, then the stabilizer $G_{\lambda}$ of $\lambda$ in $G$ is compact (that is, the coadjoint orbit $G\cdot\lambda$ is strongly elliptic). In particular, when $G_{\lambda} = K$, the holomorphic coadjoint orbit $G\cdot\lambda$ coincides with the Hermitian symmetric space $G/K$.

In the 80's, McDuff \cite{mcduff} proved that any Hermitian symmetric space of noncompact type is diffeomorphic, as a symplectic manifold, to a symplectic vector space. This yields a global version of Darboux's theorem for every Hermitian symmetric space.

Our purpose is to extend McDuff's Theorem to any holomorphic coadjoint orbit.


First of all, we introduce our symplectic model for general holomorphic coadjoint orbits of $G$, which extends the symplectic vector space model obtained for the Hermitian symmetric space case.

Let $\lambda\in\Chol$. Using the Cartan decomposition on the Lie group $G$, we know that the manifold structure of $G\cdot\lambda$ is $K$-equivariantly diffeomorphic to the product $K\cdot\lambda\times\got{p}$, on which $K$ acts diagonally. These two manifolds admit canonical symplectic structures $\Omega_{G\cdot\lambda}$ and $\Omega_{K\cdot\lambda\times\got{p}}=\Omega_{K\cdot\lambda}\oplus\Omega_{\got{p}}$ (direct product of two symplectic forms), where $\Omega_{K\cdot\lambda}$ denotes the Kirillov-Kostant-Souriau symplectic form on the compact coadjoint orbit $K\cdot\lambda$, and $\Omega_{\got{p}}$ is a $K$-invariant constant symplectic form on the vector space $\got{p}$, which definition is given in paragraph \ref{subsection:the_symplectic_forms}. 

The main goal of this article is to prove the following generalization of McDuff's Theorem, conjectured by Paradan in \cite{paradan}.

\begin{thm}
\label{thm:mainthm}
Let $\lambda\in\Chol$. Then there exists a $K$-equivariant diffeomorphism from $G\cdot\lambda$ onto $K\cdot\lambda\times\got{p}$ which takes the symplectic form $\Omega_{G\cdot\lambda}$ on $G\cdot\lambda$ to the symplectic form $\Omega_{K\cdot\lambda\times\got{p}}$ on $K\cdot\lambda\times\got{p}$.
\end{thm}

\begin{rem}
As we will see in the proof of Theorem \ref{thm:mainthm}, we can assume that the diffeomorphism, obtained in the above statement, also satisfies that each element $k\lambda\in K\cdot\lambda\subset G\cdot\lambda$ is sent to $(k\lambda,0)\in K\cdot\lambda\times\got{p}$.
\end{rem}

The symplectic manifolds $(G\cdot\lambda,\Omega_{G\cdot\lambda})$ and $(K\cdot\lambda\times\got{p},\Omega_{K\cdot\lambda\times\got{p}})$ actually have Hamiltonian $K$-manifold structures given respectively by the moment maps $\Phi_{G\cdot\lambda}:\xi\in G\cdot\lambda\subseteq\got{g}^*\mapsto \xi|_{\got{k}}\in\got{k}^*$ (that is, the standard orbit projection of $G\cdot\lambda$ on $\got{k}^*$), and
\[
\begin{array}{cccl}
\Phi_{K\cdot\lambda\times\got{p}} : & K\cdot\lambda\times\got{p} & \rightarrow & \got{k}^* \\
& (\xi,v) & \mapsto & \bigl(X\in\got{k}\mapsto\langle\xi,X\rangle + \frac{1}{2}\Omega_{\got{p}}(v,[X,v])\bigr).
\end{array}
\]
Let $\Delta_K(G\cdot\lambda):=\Phi_{G\cdot\lambda}(G\cdot\Lambda)\cap\got{t}^*_+$ and $\Delta_K(K\cdot\lambda\times\got{p}):=\Phi_{K\cdot\lambda\times\got{p}}(K\cdot\lambda\times\got{p})\cap\got{t}^*_+$ denote the associated moment polyhedra. Theorem \ref{thm:mainthm} has the following direct consequence, originally proved by Nasrin (for $\lambda$ in the center of $\got{k}^*$) and Paradan in totally different ways \cite{nasrin,paradan}.

\begin{cor}[Nasrin, Paradan]
Let $\lambda\in\Chol$. Then
\[
\Delta_K(G\cdot\lambda) = \Delta_K(K\cdot\lambda\times\got{p}).
\]
\end{cor}

This new description of the moment polyhedron $\Delta_K(G\cdot\lambda)$ allows to describe its faces, using GIT methods on the second setting. This question will be dealt with in another forthcoming paper.

This article is completely dedicated to the proof of Theorem \ref{thm:mainthm}. In section \ref{section:preliminaries}, we introduce the notion of holomorphic coadjoint orbit and some other preliminary facts about Cartan decomposition of $G$. The main tool, constructing symplectomorphisms by a Moser argument on a special noncompact setting, is given in section \ref{section:moser}. It uses the properness of the involved moment maps so as to be able to integrate a particular vector field on the noncompact manifold $K\cdot\lambda\times\got{p}$. Theorem \ref{thm:mainthm} is then proved in section \ref{section:hermitian_case_proof} for the case of Hermitian symmetric spaces, and section \ref{section:proof_mainthm} achieves the general proof.

\section{Preliminaries}
\label{section:preliminaries}

In this section, we set some definitions and notations. From now on, let $G$ be a noncompact, connected, real semisimple Lie group with finite center, and $\got{g}$ its Lie algebra.

\subsection{Properties of the Cartan decomposition}

Here, we remind of some facts about Cartan decomposition of $G$, and establish two lemmas. A good exposition on Cartan decomposition of real semisimple Lie groups can be found in \cite[Chapter VI]{knapp}.

Since $\got{g}$ is semisimple, there exists a Cartan involution $\theta$ on $\got{g}$. We recall that a Cartan involution on the Lie algebra $\got{g}$ is an involutive Lie algebra automorphism $\theta$ of $\got{g}$ such that the symmetric bilinear form $B_{\theta}$ defined by
\[
B_{\theta}(X,Y) = -B_{\got{g}}(X,\theta(Y)), \quad \text{for all }X,Y\in\got{g},
\]
is positive definite on $\got{g}$. Here, $B_{\got{g}}$ denotes the Killing form on $\got{g}$.

Let $\got{k}$ (resp. $\got{p}$) be the eigenspace of $\theta$ with eigenvalue $1$ (resp. $-1$). From the definitions, we have the Cartan decomposition $\got{g} = \got{k}\oplus\got{p}$ and the inclusions $[\got{k},\got{k}]\subseteq\got{k}$, $[\got{p},\got{p}]\subseteq\got{k}$ and $[\got{k},\got{p}]\subseteq\got{p}$. Then, $\got{k}$ is a Lie subalgebra of $\got{g}$. 

Let $K$ be the connected Lie subgroup of $G$ with Lie algebra $\got{k}$. Then, $K$ is a maximal compact subgroup of $G$, and we have a $K$-invariant diffeomorphism $K\times\got{p}\rightarrow G$, $(k,Z)\mapsto e^Zk$, known as the Cartan decomposition on the Lie group $G$.

Fix $\lambda\in\got{k}^*$ such that its stabilizer in $G$ is compact, that is, $G_{\lambda} = K_{\lambda}$. Then the Cartan decomposition induces the $K$-invariant diffeomorphism
\[
\begin{array}{cccc}
\Gamma:& K\cdot\lambda\times\got{p}& \longrightarrow& G\cdot\lambda \\
& (\xi,Z) & \longmapsto & e^Z\xi
\end{array}
\]

We can identify the tangent bundle of the homogeneous space $G\cdot\lambda$ (resp. $K\cdot\lambda$) with the manifold $G\times_{G_{\lambda}}\got{g}/\got{g}_{\lambda}$ (resp. $K\times_{K_{\lambda}}\got{k}/\got{k}_{\lambda}$) by using the diffeomorphism
\[
\begin{array}{ccc}
G\times_{G_{\lambda}}\got{g}/\got{g}_{\lambda} & \longrightarrow & T(G\cdot\lambda)\\
\phantom{}[g,X\mod\got{g}_{\lambda}] & \longmapsto & \frac{d}{dt}(g e^{tX}\lambda)|_{t=0}
\end{array}
\]
(resp. $K\times_{K_{\lambda}}\got{k}/\got{k}_{\lambda}\rightarrow T(K\cdot\lambda)$, $[k,X\mod\got{k}_{\lambda}]\mapsto \frac{d}{dt}(k e^{tX}\lambda)|_{t=0}$). Since $\Gamma$ is defined in terms of the exponential function, its derivative will involve the linear endomorphisms of $\got{g}$
\begin{equation}
\label{eq:defn_PsiZ}
\Psi_Z := \int_0^1 e^{-s\ad(Z)}ds = \sum_{n=0}^{+\infty}\frac{(-1)^n\ad(Z)^n}{(n+1)!},
\end{equation}
defined for all $Z\in\got{g}$. For $Z\in\got{g}$, the derivative of $\exp$ at $Z$ is given by the equation
\begin{equation}
\label{eq:derivative_exp}
\frac{d}{dt}\bigl(\exp(Z+tX)\bigr)|_{t=0} = \frac{d}{dt}\bigl(e^Z\exp(t\Psi_{Z}(X))\bigr)|_{t=0}, \  \forall X\in\got{g}.
\end{equation}
See for example \cite[Theorem 1.5.3]{duistermaat_kolk}. We can now compute the derivative of $\Gamma$ at any point of $K\cdot\lambda\times\got{p}$.

\begin{lem}
\label{lem:derivative_Gamma}
For all $(k\lambda,Z)\in K\cdot\lambda\times\got{p}$ and all $(X,A)\in\got{k}/\got{k}_{\lambda}\oplus\got{p}$, we have
\begin{equation}
\label{eq:relation_derivative_Gamma}
d\Gamma(k\lambda,Z)([k,X],A) = [e^Zk,X + \Ad(k^{-1})\Psi_{Z}(A)].
\end{equation}
\end{lem}

\begin{proof}
By linearity of $d\Gamma(k\lambda,Z)$, we only have to compute separately the expressions $\frac{d}{dt}(e^Zk\exp(tX)\lambda)|_{t=0}$ and $\frac{d}{dt}(\exp(Z+tA)k\lambda)|_{t=0}$. But the first term is equal to $[e^Zk,X]$ by definition. Moreover, using equation \eqref{eq:derivative_exp}, we have
\[
\frac{d}{dt}(\exp(Z+tA)k\lambda)|_{t=0} = \frac{d}{dt}(e^Zk\exp(t\Psi_{\Ad(k^{-1})Z}(\Ad(k^{-1})A)\lambda)|_{t=0},
\]
which is also equal to $[e^Zk,\Psi_{\Ad(k^{-1})Z}(\Ad(k^{-1})A)]$. And one can easily check that $\Psi_{\Ad(k^{-1})Z}(\Ad(k^{-1})A)=\Ad(k^{-1})\Psi_{Z}(A)$ from the definition of $\Psi_Z$.
\end{proof}

Now, let $Z$ be in $\got{p}$. Then $\ad(Z)$ is a symmetric endomorphism relatively to the inner product $B_{\theta}$, cf \cite[Lemma 6.27]{knapp}. Since $\Psi_Z$ is defined as the convergent series \eqref{eq:defn_PsiZ}
, it is also symmetric for $B_{\theta}$. Its eigenvalues are positive real numbers, so $\Psi_Z$ is positive definite.

Note that, for all positive integers $n$, the endomorphism $\ad(Z)^{2n}$ maps $\got{k}$ (resp. $\got{p}$) into $\got{k}$ (resp. $\got{p}$), and $\ad(Z)^{2n+1}$ maps $\got{k}$ (resp. $\got{p}$) into $\got{p}$ (resp. $\got{k}$), because of the inclusions $[\got{k},\got{k}]\subseteq\got{k}$, $[\got{p},\got{p}]\subseteq\got{k}$ and $[\got{k},\got{p}]\subseteq\got{p}$. Then, we decompose $\Psi_Z$ in two parts, $\Psi_Z = \Psi_Z^+ + \Psi_Z^-$, where
\[
\Psi_Z^+ := \sum_{n=0}^{+\infty}\frac{\ad(Z)^{2n}}{(2n+1)!} \ \text{ and } \ \Psi_Z^- :=- \sum_{n=0}^{+\infty}\frac{\ad(Z)^{2n+1}}{(2n+2)!}.
\]
These linear endomorphisms of $\got{g}$ are symmetric for $B_{\theta}$, and $\Psi_Z^+$ is also positive definite. Thus $\Psi_Z^+$ is invertible, but $\Psi_Z^-$ is not, since $\Psi_Z^-(Z) = 0$. We can define the linear map
\begin{equation}
\label{eq:defn_chiZ}
\chi_Z := \Psi_Z^-\circ(\Psi_Z^+)^{-1}:\got{g}\rightarrow\got{g}.
\end{equation}
In section \ref{section:proof_mainthm}, we will use the important property of $\chi_Z$ stated in the next lemma.

\begin{lem}
\label{lem:chi_Z_and_eigenvalues}
For all $Z\in\got{p}$, the linear map $\chi_Z:\got{g}\rightarrow\got{g}$ is symmetric for $B_{\theta}$, and its eigenvalues are in $]-1,1[$.
\end{lem}

\begin{proof}
Clearly, since $\Psi_Z^+$ is symmetric, $(\Psi_Z^+)^{-1}$ is also symmetric. Moreover, $\ad(Z)$ commutes $(\Psi_Z^+)^{-1}$, because it commutes with $\Psi_Z^+$. Now, from the definition of $\Psi_Z^-$ and the linearity of $(\Psi_Z^+)^{-1}$, $\Psi_Z^-$ obviously commutes with $(\Psi_Z^+)^{-1}$. This implies that $\chi_Z=\Psi_Z^-\circ(\Psi_Z^+)^{-1}$ is symmetric for $B_{\theta}$, and then, if $\nu_1,\ldots,\nu_r$ are the eigenvalues of $\ad(Z)$, a quick calculation shows that the eigenvalues of $\chi_Z$ are the real numbers $(e^{\nu_i}-1)/(e^{\nu_i}+1)\in]-1,1[$, for $i=1,\ldots,r$.
\end{proof}

\subsection{The holomorphic chamber $\Chol$}
\label{subsection:holomorphic_chamber}

We recall that the symmetric space $G/K$ is Hermitian if it admits a complex-manifold structure such that $G$ acts by holomorphic transformations. If $G$ and $K$ satisfies the previous hypotheses (that is, $G$ is noncompact, connected, real semisimple Lie group with finite center, and $K$ a maximal compact subgroup of $G$), then the following assertions are equivalent:
\begin{enumerate}
\item $G/K$ is Hermitian,
\item there exists $z_0$ in the center of $\got{k}$ such that $\ad(z_0)|_{\got{p}}^2 = -\id|_{\got{p}}$.
\end{enumerate}
A proof of this equivalence is given by Theorems 7.117 and 7.119 in \cite{knapp}.

Now assume $G/K$ is Hermitian, and let $z_0$ be an element of the center of $\got{k}$ such that $\ad(z_0)|_{\got{p}}^2 = -\id|_{\got{p}}$. It means that $\ad(z_0)|_{\got{p}}$ defines a $K$-invariant $\C$-vector space structure on $\got{p}$. Denote by $\got{p}_{\C}$ the complexification of $\got{p}$, and similarly $\got{g}_{\C}$ and $\got{k}_{\C}$. The linear action of $K$ on $\got{p}$, defined by the adjoint action $\Ad$, induces a complex-linear action of $K$ on $\got{p}_{\C}$.

Denote by $\got{p}^{\pm,z_0}$ the eigenspace $\ker(\ad(z_0)|_{\got{p}_{\C}}\mp i)$ of $\ad(z_0)|_{\got{p}_{\C}}$ associated to the eigenvalue $\pm i$. Especially, $\ad(z_0)$ is multiplication by the complex number $\pm i$ on $\got{p}^{\pm,z_0}$. These two subspaces of $\got{p}_{\C}$ are $K$-stable.

Let $T$ be a maximal torus of the connected compact group $K$. We set the following convention: an element $\alpha\in\got{t}^*$ is a \emph{root} of $\got{g}$ (resp. $\got{k}$) if there exists $X\in\got{g}_{\C}$ (resp. $X\in\got{k}_{\C}$), $X\neq 0$, such that $[H,X] = i\alpha(H)X$ for all $H\in\got{t}$. The associated root space is
\[
\got{g}_{\alpha} := \{X\in\got{g}_{\C}\ |\ [H,X] = i\alpha(H)X, \forall H\in\got{t}\}.
\]
If $\alpha$ is a root of $\got{g}$, then either $\got{g}_{\alpha}\subseteq\got{k}_{\C}$ ($\alpha$ is said \emph{compact root}), or $\got{g}_{\alpha}\subseteq\got{p}_{\C}$ (\emph{noncompact root}). Note that the compact roots are the roots of the Lie algebra $\got{k}$. The set of compact (resp. noncompact) roots is denoted by $\got{R}_c$ (resp. $\got{R}_n$). Fix once and for all $\got{t}^*_+$ a Weyl chamber of $K$ in $\got{t}^*$, and let $\got{R}_c^+$ be the system of positive compact roots associated to this Weyl chamber. Notice that, since $z_0\in\got{t}$, for any noncompact root $\beta$, we have either $\got{g}_{\beta}\subseteq\got{p}^{+,z_0}$ (\emph{positive noncompact roots}) or $\got{g}_{\beta}\subseteq\got{p}^{-,z_0}$ (\emph{negative noncompact roots}). Denote by $\got{R}_n^{+,z_0}$ the set of positive noncompact roots of $\got{g}$. Then $\got{R}_c^+\cup\got{R}_n^{+,z_0}$ is a system of positive roots of $\got{g}$. Indeed, we can easily see that for all $\alpha\in\got{R}_c^+$, we have $\alpha(z_0)=0$, and, for all $\beta\in\got{R}_n^{+,z_0}$, $\beta(z_0) = 1$.

\begin{defn}
The \emph{holomorphic chamber} is the subchamber of $\got{t}^*_+$ defined by
\[
\Chol^{z_0}:=\{\xi\in\got{t}^*\ | \ (\beta,\xi)>0, \forall\beta\in\got{R}_n^{+,z_0}\},
\]
where $(\cdot,\cdot)$ is the inner product on $\got{t}^*$ induced by $B_{\theta}$. A \emph{holomorphic coadjoint orbit} is a coadjoint orbit $\Orb$ of $G$ which intersects $\Chol^{z_0}$ on a nonempty set.
\end{defn}

Let $\lambda\in\Chol^{z_0}$. The holomorphic coadjoint orbit $G\cdot\lambda$ has a natural $G$-invariant K\"ahler structure:
\begin{enumerate}
\item a canonical $G$-invariant symplectic form $\Omega_{G\cdot\lambda}$, called the Kirillov-Kostant-Souriau symplectic structure on $G\cdot\lambda$;
\item a $G$-invariant complex structure $J_{G\cdot\lambda}$, which holomorphic tangent bundle $T^{1,0}(G\cdot\lambda)\rightarrow G\cdot\lambda$ is equal, above $\lambda$, to the $T$-submodule
\[
\sum_{\alpha\in\got{R}_c^+, (\alpha,\lambda)\neq 0}\got{g}_{\alpha} + \underbrace{\sum_{\beta\in\got{R}_n^-}\got{g}_{\beta}}_{\got{p}^-}
\]
One can check that this complex structure is compatible with the symplectic form $\Omega_{G\cdot\lambda}$.
\end{enumerate}
Besides, the stabilizer of $\lambda$ is clearly compact, since $(\beta,\lambda)\neq 0$ for all $\beta\in\got{R}_n$.

\begin{rem}
In the rest of the paper, we will omit the notation $z_0$ in $\Chol^{z_0}$, and write $\Chol$ instead.
\end{rem}

\subsection{The symplectic forms $\Gamma^*\Omega_{G\cdot\lambda}$ and $\Omega_{K\cdot\lambda\times\got{p}}$}
\label{subsection:the_symplectic_forms}

Let $\lambda\in\Chol$. The manifold $K\cdot\lambda\times\got{p}$ can be identified to the holomorphic coadjoint orbit $G\cdot\lambda$ through the diffeomorphism $\Gamma$. This gives $K\cdot\lambda\times\got{p}$ a first canonical symplectic structure $\Gamma^*\Omega_{G\cdot\lambda}$, where $\Omega_{G\cdot\lambda}$ is the Kirillov-Kostant-Souriau symplectic form on $G\cdot\lambda$. Note that, here, $\Omega_{G\cdot\lambda}$ is defined by the formula
\[
(\Omega_{G\cdot\lambda})|_{g\lambda}\bigl([g,(X,A)],[g,(Y,B)]\bigr) = \langle\lambda,[X,Y]\rangle + \langle\lambda,[A,B]\rangle
\]
for all $g\in G$, and all $(X,A)$, $(Y,B)\in\got{k}/\got{k}_{\lambda}\oplus\got{p}$, where $\langle\cdot,\cdot\rangle:\got{g}^*\times\got{g}\rightarrow\R$ denotes the standard pairing. Then, from Lemma \ref{lem:derivative_Gamma}, the expression of $\Gamma^*\Omega_{G\cdot\lambda}$ is given for all $(k\lambda,Z)\in K\cdot\lambda\times\got{p}$ and all $(X,A)$, $(Y,B)\in\got{k}/\got{k}_{\lambda}\oplus\got{p}$ by
\begin{multline*}
(\Gamma^*\Omega_{G\cdot\lambda})|_{(k\lambda,Z)}\left(([k,X],A),([k,Y],B)\right) \\
= \langle\lambda,[X+\Ad(k^{-1})\Psi_Z(A), Y+\Ad(k^{-1})\Psi_Z(B)]\rangle.
\end{multline*}
We can also split the right-hand term up using the operators $\Psi_Z^+$ and $\Psi_Z^-$:
\begin{multline}
\label{eq:formula_GammaOmegaGlambda}
(\Gamma^*\Omega_{G\cdot\lambda})|_{(k\lambda,Z)}\left(([k,X],A),([k,Y],B)\right) \\
= \left\langle\lambda,[X+\Ad(k^{-1})\Psi_Z^-(A), Y+\Ad(k^{-1})\Psi_Z^-(B)]\right\rangle \\
+ \left\langle  k\lambda,[\Psi_Z^+(A),\Psi_Z^+(B)]\right\rangle. 
\end{multline}

Now, consider the following $K$-invariant symplectic form on $\got{p}$
\begin{equation}
\label{eq;defn_Omegap}
\Omega_{\got{p}}(A,B) := B_{\theta}(A,\ad(z_0)B), \quad \forall A,B\in\got{p}.
\end{equation}
It is symplectic since $\ad(z_0)|_{\got{p}}^2 = -\id_{\got{p}}$ and $B_{\theta}$ is nondegenerate on $\got{p}$. The $K$-invariance of $\Omega_{\got{p}}$ is provided by the one of $B_{\theta}$ and the fact that $z_0$ is centralized by $K$. Denoting $\Omega_{K\cdot\lambda}$ the Kirillov-Kostant-Souriau symplectic form on the compact coadjoint orbit $K\cdot\lambda$, we thus have another canonical $K$-invariant symplectic structure $\Omega_{K\cdot\lambda\times\got{p}} := \Omega_{K\cdot\lambda}\oplus\Omega_{\got{p}}$, that is, the symplectic structure obtained from the direct product of the symplectic manifolds $(K\cdot\lambda,\Omega_{K\cdot\lambda})$ and $(\got{p},\Omega_{\got{p}})$. This new symplectic form is given by the formula
\begin{multline}
\label{eq:formula_OmegaKlambdaxp}
\Omega_{K\cdot\lambda\times\got{p}}|_{(k\lambda,Z)}\left(([k,X],A),([k,Y],B)\right) \\
= \left\langle\lambda,[X, Y]\right\rangle + B_{\theta}(A,[z_0,B])
\end{multline}
for all $(k\lambda,Z)\in K\cdot\lambda\times\got{p}$ and all $(X,A)$, $(Y,B)\in\got{k}/\got{k}_{\lambda}\oplus\got{p}$.

Our purpose boils down to prove the existence of a symplectomorphism between the two symplectic manifolds $(K\cdot\lambda\times\got{p},\Gamma^*\Omega_{G\cdot\lambda})$ and $(K\cdot\lambda\times\got{p},\Omega_{K\cdot\lambda\times\got{p}})$.

\section{A noncompact version of Moser's theorem}
\label{section:moser}

We first need a tool allowing us to prove that two symplectic manifolds are diffeomorphic. In general, a Moser argument \cite{moser} is an effective way of solving such a problem. However, in the noncompact setting, it is not always possible to apply a Moser argument because not all vector fields are completely integrable. Fortunately, this is still possible in good cases. For instance, in \cite{mcduff}, the proof uses the geodesic completeness of the manifold, combined with Rauch's comparison theorem. Another method is to work on Hamiltonian manifolds with proper moment maps, as in \cite{karshon_tolman}. Here, we propose to follow this second method.

Let $K$ be a connected compact Lie group, $V$ a finite dimensional real representation of $K$, and $M$ a connected compact manifold, endowed with an action of $K$. This induces a diagonal action of $K$ on the trivial vector bundle $M\times V$. Let $(\Omega_t)_{t\in[0,1]}$ be a smooth family of $K$-invariant symplectic forms on $M\times V$, such that, for all $t\in[0,1]$, the symplectic manifold $(M\times V,\Omega_t)$ admits a structure of Hamiltonian $K$-manifold given by the moment map $\phi_t:M\times V\rightarrow\got{k}^*$.

In the next statement, $\left(T_{(m,0)}(\{m\}\times V)\right)^{\Omega_t}$ denotes the symplectic orthogonal complement of the tangent space $T_{(m,0)}(\{m\}\times V)$ above the point $(m,0)$ of the submanifold $\{m\}\times V$, in the ambient tangent space $T_{(m,0)}(M\times V)$, relatively to the symplectic form $\Omega_t|_{(m,0)}$.

\begin{thm}
\label{thm:moserthm}
Assume that the following assertions are satisfied:
\begin{enumerate}
\item there exists a smooth family $(\mu_t)_{t\in[0,1]}$ of $K$-invariant $1$-forms such that $d\mu_t = \frac{d}{dt}\Omega_t$ for all $t\in[0,1]$,
\item the set $\{\phi_t(m,0); m\in M, t\in[0,1]\}$ is bounded in $\got{k}^*$,
\item for all $m\in M$ and all $t\in[0,1]$,
\[
\left(T_{(m,0)}(\{m\}\times V)\right)^{\Omega_t} = T_{(m,0)}(M\times\{0\}),
\]
\item there exists two positive numbers $d$ and $\gamma$ such that
\[
\|\phi_t(m,v)\| \geqslant d\|v\|^{\gamma}, \quad \forall (m,v)\in M\times V,\ \forall t\in[0,1],
\]
\end{enumerate}
then, there exists a $K$-invariant isotopy $\rho_t:M\times V\rightarrow M\times V$ such that $\rho_t^*\Omega_t = \Omega_0$ for all $t\in[0,1]$.

Moreover, if, for some $m_0\in M$, we have $\mu_t|_{(m_0,0)}(u,0)=0$ for all $(t,u)\in[0,1]\times T_{m_0}M$, then $\rho_t(m_0,0) = (m_0,0)$ for all $t\in[0,1]$.
\end{thm}

The idea is to integrate the $K$-invariant time-dependent vector field $\xi_t$ defined on $M\times V$ by
\begin{equation}
\label{eq:defn_xi_t}
\imath(\xi_t)\Omega_t = -\mu_t, \quad \forall t\in[0,1].
\end{equation}
Assertion 1) ensures that the isotopy $\rho_t$, obtained by integrating $\xi_t$, verifies the equation $\rho_t^*\Omega_t = \Omega_0$. The other three conditions are used to make $\xi_t$ completely integrable on the noncompact manifold $M\times V$. Note that the last condition implies the properness of $\phi_t$.

\begin{lem}
\label{lem:exactness_1formsfamily}
Let $(\mu_t)_{t\in[0,1]}$  be a smooth family of $K$-invariant $1$-forms on $M\times V$. There exists a smooth family $(f_t)_{t\in[0,1]}$ of $K$-invariant $C^{\infty}$-functions on $M\times V$, such that
\begin{enumerate}
\item[(\emph{i})] $df_t|_{T(M\times\{0\})} \equiv 0$,
\item[(\emph{ii})] $\imath(v)(\mu_t-df_t) = 0$ on $M\times\{0\}$ for all $v\in V$.
\end{enumerate}
\end{lem}

\begin{proof}
Take $f_t(m,v):=2\int_0^1\mu_t|_{(m,sv)}(0,sv)ds$ for all $t\in[0,1]$ and all $(m,v)\in M\times V$.
\end{proof}

\begin{proof}[Proof of Theorem \ref{thm:moserthm}]
By Lemma \ref{lem:exactness_1formsfamily}, without loss of generality we may assume that, for all $t\in[0,1]$, the $1$-form $\mu_t$ is $K$-invariant and verifies
\begin{equation}
\label{eq:proof_thm3.1_hypothesis_on_mu}
\mu_t|_{(m,0)}(0,v) = 0, \quad\forall (m,v)\in M\times V.
\end{equation}
%
From \eqref{eq:defn_xi_t} and \eqref{eq:proof_thm3.1_hypothesis_on_mu}, we notice that, for all $(m,v)\in M\times V$, the tangent vector $\xi_t(m,0)$ is in the space $\left(T_{(m,0)}(\{m\}\times V)\right)^{\Omega_t}$. But, hypothesis 3) implies that $\xi_t(m,0)$ is in $T_{(m,0)}(M\times\{0\})$ for all $m\in M$ and all $t\in[0,1]$.

In order to prove Theorem \ref{thm:moserthm}, we must integrate this vector field and obtain an isotopy on $M\times V$. We thus have to consider the time-dependent differential equation
\begin{equation}
\label{eq:differential_equation}
\left\{\begin{array}{l}
\rho(x) = x, \\
\frac{d}{dt}\rho_t(x) = \xi_t(\rho_t(x)),
\end{array}\right.
\end{equation}
for any initial condition $x\in M\times V$. By Cauchy-Lipschitz's theorem, the domain of definition $\mathcal{D}\subseteq[0,1]\times M\times V$ of the integral curves of \eqref{eq:differential_equation} is an open set.

Let $r$ be any positive real number, and $U_r:=M\times B(0,r)$ the open connected neighborhood of $M\times\{0\}$ in $M\times V$, where $B(0,r)$ is the open ball in $V$ centered at $0$ with radius $r$, defined for some $K$-invariant inner product on $V$. The closure $\overline{U}_r$ is the compact subset $M\times\overline{B}(0,r)$ of $M\times V$. We define the segment
\[
I_r:=\{\varepsilon\in[0,1] \ | \ \forall x\in\overline{U}_r, \text{ the curve $\rho_t(x)$ is defined for all } t\in[0,\varepsilon]\}.
\]
It is nonempty since clearly $0\in I_r$. We shall prove the equality $I_r=[0,1]$. But $[0,1]$ is connected, so it is enough to prove that $I_r$ is open and closed in $[0,1]$.

The openness of $I_r$ is induced directly by the one of $\mathcal{D}$ and the compactness of $\overline{U}_r$. Indeed, if $\varepsilon\in I_r$, then for each $x\in\overline{U}_r$, there exists a neighborhood $\mathscr{V}_x$ of $x$ in $M\times V$ and a real number $\varepsilon_x>\varepsilon$ such that $\rho_t(y)$ is defined for all $y\in\mathscr{V}_x$ and all $t\in[0,\varepsilon_x[$, since $\mathcal{D}$ is open in $[0,1]\times M\times V$. But $\overline{U}_r$ is compact, then there exists a finite number of points $x_1,\ldots,x_s$ in $\overline{U}_r$ such that the family of neighborhoods $(\mathscr{V}_{x_i})_{i=1,\ldots,s}$ is a covering of $\overline{U}_r$. This implies that the interval $[0,\min_{i=1}^s\varepsilon_i[$, which contains $\varepsilon$, is included in $I_r$. Hence $I_r$ is open.

We now show that $I_r$ is closed in $[0,1]$. First, we have to prove three claims.


\begin{claim}
\label{claim:moserthm_claim1}
For all $m\in M$, the integral curve $t\mapsto\rho_t(m,0)$ is defined for all $t\in[0,1]$.
\end{claim}

Since $\xi_t(m,0)$ is tangent to the submanifold $M\times\{0\}$ for all $m\in M$ and $t\in[0,1]$, the integral curve $t\mapsto\rho_t(m,0)$ is included in $M\times\{0\}$ for all $m\in M$. But $M$ is compact. Hence, the maximal integral curve of \eqref{eq:differential_equation}, starting from any point $(m,0)\in M\times\{0\}$, is defined for all $t\in[0,1]$.

\bigskip

Let $m_0$ be an element of $M$. This point enables to define the family of vectors
\[
c_t:=\phi_t\circ\rho_t(m_0,0)-\phi_0(m_0,0) \in\got{k}^*, \quad \forall t\in[0,1],
\]
and the constant
\[
C := \sup_{t\in[0,1]}\|c_t\|<+\infty.
\]
This supremum is finite because of hypothesis 2). We also define the real numbers
\[
D_r:=\sup_{x\in\overline{U}_r}\|\phi_0(x)\|, \quad \forall r>0.
\]
We clearly have $D_r\leq D_s$ for all positive numbers $r\leq s$.


\begin{claim}
\label{claim:moserthm_claim2}
Let $r>0$ and $\varepsilon\in\,]0,1]$ such that, for all $x\in U_r$, the integral curve $\rho_t(x)$ is defined for all $t\in[0,\varepsilon[$. Then, for any $t\in[0,\varepsilon[$, we have 
\begin{equation}
\label{eq:moserthm_claim2_inclusion}
\rho_t(U_r)\subset\overline{U}_{\left(\frac{D_{r}+C}{d}\right)^{1/\gamma}}.
\end{equation}
\end{claim}

For $r>0$ and $\varepsilon\in\,]0,1]$ satisfying the hypothesis of Claim \ref{claim:moserthm_claim2}, we obtain for all $t\in[0,\varepsilon[$ a smooth map $\rho_t:U_r\rightarrow M\times V$. Since $\Omega_t$ is a closed $2$-form, $L_{\xi_t}\Omega_t = d(\imath(\xi_t)\Omega_t) = -d\mu_t = -\frac{d}{dt}\Omega_t$ by Cartan's Formula. Thus, $0=\rho_t^*(L_{\xi_t}\Omega_t+\frac{d}{dt}\Omega_t) = \frac{d}{dt}(\rho_t\Omega_t)$ on $U_r$, for all $t\in[0,\varepsilon[$. As a result, we have
\begin{equation}
\label{eq:moserthm_symplforms_equality_on_Ur}
\rho_t^*\Omega_t = \Omega_0|_{U_r}, \quad \forall t\in[0,\varepsilon[.
\end{equation}
Note that $U_r$ is a $K$-invariant neighborhood of $M\times\{0\}$ in $M\times V$. Therefore, $\phi_0|_{U_r}:U_r\rightarrow\got{k}^*$ is a moment map for the symplectic $K$-manifold $(U_r,\Omega_0|_{U_r})$. But, by \eqref{eq:moserthm_symplforms_equality_on_Ur}, $\rho_t^*\phi_t=\phi_t\circ\rho_t:U_r\rightarrow\got{k}^*$ is another moment map of $(U_r,\Omega_0|_{U_r})$. We deduce from the connectedness of $U_r$ that, for all $t\in[0,\varepsilon[$, $\phi_t\circ\rho_t-\phi_0$ is a constant map and, more precisely, we have the equality $\phi_t\circ\rho_t(x)=\phi_0(x)+c_t$ for all $x\in U_r$ and all $t\in[0,\varepsilon[$. In particular, this induces the inequality
\begin{equation}
\label{eq:majoration_phi_t_rho_t}
\|\phi_t(\rho_t(x))\| \leq \|\phi_0(x)\| + \|c_t\| \leq D_r + C, \quad \forall x\in U_r, \forall t\in[0,\varepsilon[.
\end{equation}

Denote by $\pi_V:M\times V\rightarrow V$ the canonical projection. From hypothesis 4) and \eqref{eq:majoration_phi_t_rho_t}, we have
\[
D_r+C \geq \|\phi_t(\rho_t(x))\| \geq d\|\pi_V(\rho_t(x))\|^{\gamma}, \quad \forall x\in U_r, \forall t\in[0,\varepsilon[.
\]
that is, the inclusion $\rho_t(U_r)\subseteq\overline{U}_{\left(\frac{D_{r}+C}{d}\right)^{1/\gamma}}$.


\begin{claim}
\label{claim:moserthm_claim3}
Let $r>0$ and $\varepsilon\in]0,1]$ such that, for all $x\in\overline{U}_r$, the integral curve $\rho_t(x)$ is defined for all $t\in[0,\varepsilon[$. Then, for all $t\in[0,\varepsilon[$, we have
\begin{equation}
\label{eq:moserthm_claim3_inclusion}
\rho_t(\overline{U}_r)\subseteq\overline{U}_{\left(\frac{D_{r+1}+C}{d}\right)^{1/\gamma}}.
\end{equation}
\end{claim}

Let $\tau$ be in $[0,\varepsilon[$. Then, $\tau\in I_r$ and, since $\mathcal{D}$ is open, as in the proof that $I_r$ is open, we can easily check that there exists an open neighborhood $\mathscr{V}$ of $\overline{U}_r$ such that, for all $y\in\mathscr{V}$, the integral curve $\rho_t(y)$ is defined for all $t\in[0,\tau]$. Now, $\overline{U}_r$ is compact, thus by a standard topological argument, there exists a real number $r'>r$ such that $\overline{U}_{r'}\subseteq\mathscr{V}$. We can assume that $r<r'<r+1$, and consequently, $D_r\leq D_{r'}\leq D_{r+1}$. A direct application of Claim 2 on $U_{r'}$ and $[0,\tau[$ yields the inclusion
\[
\rho_t(U_{r'})\subseteq\overline{U}_{\left(\frac{D_{r'}+C}{d}\right)^{1/\gamma}}.
\]
Finally, \eqref{eq:moserthm_claim3_inclusion} results from the obvious inclusions $\rho_t(\overline{U}_r) \subseteq \rho_t(U_{r'})$ and $\overline{U}_{\left(\frac{D_{r'}+C}{d}\right)^{1/\gamma}} \subseteq \overline{U}_{\left(\frac{D_{r+1}+C}{d}\right)^{1/\gamma}}$.

\bigskip

Now we are able to show that $I_r$ is closed. Actually, it is enough to prove that $\varepsilon_r:=\sup I_r$ is in the interval $I_r$. By definition of $I_r$, the integral curves $\rho_t(x)$ are defined for all $t\in[0,\varepsilon_r[$, for all $x\in\overline{U}_r$. But, from Claim 3, we have $\rho_t(\overline{U}_r)\subseteq\overline{U}_{\left(\frac{D_{r+1}+C}{d}\right)^{1/\gamma}}$ for all $t\in[0,\varepsilon[$, the second set being compact. Hence, we can extend, in $t=\varepsilon_r$, each integral curve with initial condition in $\overline{U}_r$, that is, $\varepsilon_r\in I_r$, and $I_r$ is a closed subset of $[0,1]$.

Finally, since $I_r$ is open and closed in the connected set $[0,1]$, we have $I_r = [0,1]$, and, thus, every integral curve with initial condition in $\overline{U}_r$ is completely integrable. Therefore, the time-dependent vector filed $\xi_t$ is complete in $M\times V = \cup_{r>0}\overline{U}_r$. It defines, for all $t\in[0,1]$, a map $\rho_t:M\times V\mapsto M\times V$, which is a diffeomorphism onto \cite[Theorem 52]{lafontaine}, and $K$-equivariant because $\xi_t$ is $K$-invariant. This proves that $\rho_t$ is a $K$-equivariant isotopy of $M\times V$ that verifies the equality $\rho_t^*\Omega_t = \Omega_0$ for all $t\in[0,1]$.

It remains to prove the last assertion. We assumed, at the beginning of the proof, that $(\mu_t)_{t\in[0,1]}$ is a smooth family of $K$-invariant $1$-forms such that $\mu_t|_{(m,0)}(0,v) = 0$ for all $(m,v)\in M\times V$. But now, the hypothesis on $m_0$ yields $\mu_t|_{(m_0,0)}\equiv 0$, and then $\xi_t(m_0,0)=0$, for all $t\in[0,1]$. By uniqueness of the maximal integral curve of $\xi_t$ with initial condition $(m_0,0)$, we conclude that $\rho_t(m_0,0)=(m_0,0)$ for all $t\in[0,1]$. This completes the theorem's proof .
\end{proof}

Theorem \ref{thm:moserthm} works for every smooth family $(\Omega_t)_{t\in[0,1]}$ of symplectic forms, not only for segments as in the classical Moser argument. But, it is generally difficult to find moment maps for arbitrary symplectic paths. So, for practical purposes, we study segments of symplectic forms as soon as possible.

Note that the manifold we are studying here is the trivial vector bundle $M\times V$. Thus we can define the two canonical maps
\[
i:M\hookrightarrow M\times V \quad \text{and}\quad \pi_M:M\times V\twoheadrightarrow M.
\]
The map
\[
F:M\times V\times[0,1]\rightarrow M\times V, (m,v,t)\mapsto (m,tv)
\]
is a $K$-equivariant homotopy of $M\times V$ such that $F(m,v,0)=i\circ\pi_M(m,v)=(m,0)$ and $F(m,v,1)=\id_{M\times V}(m,v)$. For any $2$-form $\omega$ we define the $1$-form $h_F(\omega)$ at $(m,v)\in M\times V$ by the formula $h_F(\omega)|_{(m,v)} := \int_{t=0}^{t=1}(F^*\omega)|_{(m,v,t)}$. By the Poicar\'e Lemma, we have
\begin{equation}
\label{eq:formula_PoincareLemma}
d\circ h_F+h_F\circ d = \id_{M\times V}^* - \pi_M^*\circ i^*,
\end{equation}
see for example \cite{warner}. One can easily check from the definition that, if $\omega$ is $K$-invariant, then $h_F(\omega)$ is. Moreover, the $1$-form $h_F(\omega)$ vanishes on the submanifold $M\times\{0\}$.

\begin{cor}
\label{cor:moserthm_segment_and_poincarecondition}
Let $\Omega_0$ and $\Omega_1$ be two $K$-invariant symplectic forms on $M\times V$, with moment maps $\phi_0$ and $\phi_1$ respectively. Set $\Omega_t=t\Omega_1+(1-t)\Omega_0$ and $\phi_t = t\phi_1+(1-t)\phi_0$ for all $t\in[0,1]$. If the following assertions are satisfied,
\begin{enumerate}
\item[a)] for all $t\in[0,1]$, $\Omega_t$ is symplectic on $M\times V$,
\item[b)] the $2$-form $\Omega_1-\Omega_0$ is in the kernel of the linear operator $i^*$,
\item[c)] for all $m\in M$ and all $t\in[0,1]$,
\[
\left(T_{(m,0)}(\{m\}\times V)\right)^{\Omega_t} = T_{(m,0)}(M\times\{0\}),
\]
\item[d)] there exists two positive numbers $d$ and $\gamma$ such that
\[
\|\phi_t(m,v)\| \geqslant d\|v\|^{\gamma}, \quad \forall (m,v)\in M\times V,\ \forall t\in[0,1],
\]
\end{enumerate}
then, there exists a $K$-equivariant symplectomorphism from $(M\times V,\Omega_0)$ onto $(M\times V,\Omega_1)$ fixing each element of $M\times\{0\}$.
\end{cor}

\begin{proof}
Let $\mu$ be the $1$-form defined by $\mu=h_F(\Omega_1-\Omega_0)$. Hypothesis b) and formula \eqref{eq:formula_PoincareLemma} yield that $d\mu = \Omega_1-\Omega_0$, because $\Omega_1-\Omega_0$ is a closed $2$-form. Thus assertion 1) of Theorem \ref{thm:moserthm} is satisfied. It remains to prove that assertion 2) is also satisfied. But, seeing that the family of moment maps $(\phi_t)_{t\in[0,1]}$ depends continuously on $t$, the set $\{\phi_t(m,0); m\in M, t\in[0,1]\}$ is clearly compact in $\got{k}^*$. Now assertion 2) is verified, so we can conclude the proof, applying Theorem \ref{thm:moserthm} with $\mu=h_F(\Omega_1-\Omega_0)$ which vanishes on $M\times\{0\}$, so that we get a symplectomorphism fixing each point of $M\times\{0\}$.
\end{proof}

\begin{exmp}
\label{exmp:vector_space_example_for_moserthm}
Assume that $(V,\omega^1)$ is a Hamiltonian $K$-manifold, with moment map $\phi^1$ such that there exists two positive real numbers $d$ and $\gamma$ satisfying the assertion:
\begin{equation}
\label{eq:exmp_moserthm_strongpropercondition}
\|\phi^1(v)\|\geq d\|v\|^{\gamma}, \quad \forall v\in V.
\end{equation}
Let $\delta$ be a positive real number, and $\omega^{\delta} := \delta\omega^1$. For every $t\in[0,1]$, the $2$-form $\omega_t:=t\omega^1+(1-t)\omega^{\delta}=(t+(1-t)\delta)\omega^1$ is symplectic, and $\phi_t:=(t+(1-t)\delta)\phi^1$ is a moment map for the Hamiltonian $K$-manifold $(V,\omega_t)$. Since $M$ is a single point here, hypothesis a), b) and c) of Corollary \ref{cor:moserthm_segment_and_poincarecondition} are clearly satisfied, and assertion d) is given by
\[
\|\phi_t(v)\| = (t+(1-t)\delta)\|\phi^1(v)\|\geq \min\{1,\delta\}d\|v\|^{\gamma}, \quad \forall v\in V
\]
Therefore, there exists a $K$-equivariant diffeomorphism from $V$ onto $V$ which takes $\omega^1$ to $\omega^{\delta}$.
\end{exmp}

\begin{rem}
Note that, in Example \ref{exmp:vector_space_example_for_moserthm}, instead of \eqref{eq:exmp_moserthm_strongpropercondition}, we could have assumed that $\phi^1$ is only proper. Indeed, in the proof of Theorem \ref{thm:moserthm}, apply the inequality $\|\phi_t(v)\| \geq \min\{1,\delta\}\|\phi^1(v)\|$ in \eqref{eq:majoration_phi_t_rho_t}, which yields that $\rho_t(U_r)$ is included in the compact set $(\phi^1)^{-1}\bigl(\overline{U}_{\frac{D_r+C}{\min\{1,\delta\}}}\bigr)$, for all $t\in[0,\varepsilon[$. With an obvious change in Claim \ref{claim:moserthm_claim3}, this proves that, in this situation, $I_r$ is also equal to the segment $[0,1]$.
\end{rem}

\section{Proof of the Hermitian symmetric space case}
\label{section:hermitian_case_proof}

In this section, we prove Theorem \ref{thm:mainthm} when the holomorphic coadjoint orbit is $G\cdot\lambda_0$, where $\lambda_0$ is the element of $\got{t}^*$ identified with $z_0$ using the inner product $B_{\theta}$ on $\got{g}$. The element $\lambda_0$ is actually in $\Chol$ because, for all noncompact positive root $\beta$, we have $\beta(z_0)=1$. The Hermitian symmetric space $G/K$ coincides with the coadjoint orbit $G\cdot\lambda_0$ since $\lambda_0$ is centralized by $K$. The diffeomorphism $\Gamma$ is expressed here by the map
\[
\begin{array}{cccl}
\Gamma_0:&\got{p}&\longrightarrow& G\cdot\lambda_0 \\
& Z & \longmapsto & e^Z\lambda_0.
\end{array}
\]
The symplectic form $\Gamma_0^*\Omega_{G\cdot\lambda_0}$, given by the formula \eqref{eq:formula_GammaOmegaGlambda} in the general case, is now
\[
(\Gamma_0^*\Omega_{G\cdot\lambda_0})|_Z(A,B) = \langle\lambda_0,[\Psi_Z^+(A),\Psi_Z^+(B)]\rangle, \quad \forall A,B\in\got{p}.
\]

\begin{thm}[McDuff]
\label{thm:mainthm_G/Kcase}
There exists a $K$-equivariant diffeomorphism from manifold $G\cdot\lambda_0$ onto $\got{p}$ which takes the symplectic form $\Omega_{G\cdot\lambda_0}$ on $G\cdot\lambda_0$ to the symplectic form $\Omega_{\got{p}}$ on $\got{p}$, such that $\lambda_0\in G\cdot\lambda_0$ is sent to $0\in\got{p}$.
\end{thm}

We present here a completely different proof of this result, using Theorem \ref{thm:moserthm}. The main difference is that we need proper moment maps on our Hamiltonian $K$-manifolds.

The canonical projection map $G\cdot\lambda_0\subset\got{g}^*\rightarrow\got{k}^*$ is known to be a moment map of the Hamiltonian $K$-manifold $(G\cdot\lambda_0,\Omega_{G\cdot\lambda_0})$. Composing with $\Gamma_0$, we get a moment map $\Phi_{\Gamma_0^*\Omega_{G\cdot\lambda_0}}$ for the Hamiltonian $K$-manifold $(\got{p},\Gamma_0^*\Omega_{G\cdot\lambda_0})$. This moment map is defined by
\[
\begin{array}{cccl}
\Phi_{\Gamma_0^*\Omega_{G\cdot\lambda_0}} : & \got{p} & \rightarrow & \got{k}^* \\
& Z & \mapsto & \bigl(X\in\got{k}\mapsto\langle e^Z\lambda_0,X\rangle\bigr).
\end{array}
\]

\begin{lem}
\label{lem:momentmap_Phi_Gamma0OmegaGlambda0_is_proper}
For all $Z\in\got{p}$, we have
\[
\langle\Phi_{\Gamma_0^*\Omega_{G\cdot\lambda_0}}(Z)-\lambda_0,z_0\rangle \geqslant \frac{1}{2}\|Z\|^2.
\]
In particular, the moment map $\Phi_{\Gamma_0^*\Omega_{G\cdot\lambda_0}} :\got{p}\rightarrow\got{k}^*$ is proper.
\end{lem}

\begin{proof}
First notice that, for all $Z\in\got{p}$ and all $X\in\got{k}$, we have
\[
\langle \lambda_0,e^{-\ad(Z)}X\rangle = \Bigl\langle\lambda_0,\sum_{k=0}^{\infty}\frac{\ad(-Z)^{2k}}{(2k)!}X\Bigr\rangle = B_{\theta}\left(z_0,\sum_{k=0}^{\infty}\frac{\ad(Z)^{2k}}{(2k)!}X\right).
\]
But, $\ad(Z)$ is symmetric for the inner product $B_{\theta}$. Thus, we get
\begin{align*}
\langle e^Z\lambda_0-\lambda_0,z_0\rangle & = \sum_{k=1}^{\infty}\frac{1}{(2k)!}B_{\theta}(\ad(Z)^k z_0,\ad(Z)^k z_0) \\
& \geqslant \frac{1}{2}B_{\theta}([z_0,Z],[z_0,Z]) = \frac{1}{2}\|Z\|^2.
\end{align*}
Consequently, the map $\Phi_{\Gamma_0^*\Omega_{G\cdot\lambda_0}}-\lambda_0:\got{p}\rightarrow\got{k}^*$ is proper, and so is $\Phi_{\Gamma_0^*\Omega_{G\cdot\lambda_0}}$.
\end{proof}

As for the second symplectic form $\Omega_{\got{p}}$ on $\got{p}$, it is a constant symplectic form on a symplectic vector space. Recall that it is defined by \eqref{eq;defn_Omegap}. Therefore, one can easily check that a moment map for $(\got{p},\Omega_{\got{p}})$ is
\[
\begin{array}{cccl}
\Phi_{\Omega_{\got{p}}}:&\got{p}&\rightarrow&\got{k}^*\\
&Z&\mapsto&\bigl(X\in\got{k}\mapsto\langle\lambda_0,[[X,Z],Z]\rangle\bigr).
\end{array}
\]

\begin{lem}
\label{lem:momentmap_Phi_Omega_gotp_is_proper}
We have $\langle\Phi_{\Omega_{\got{p}}}(Z),z_0\rangle = \|Z\|^2$, for all $Z\in\got{p}$. In particular, the moment map $\Phi_{\Omega_{\got{p}}} :\got{p}\rightarrow\got{k}^*$ is proper.
\end{lem}

\begin{proof}
From the definition of $\Phi_{\Omega_{\got{p}}}$, we obtain the following equalities,
\[
\langle\Phi_{\Omega_{\got{p}}}(Z),z_0\rangle = B_{\theta}(z_0,[[z_0,Z],Z]) = B_{\theta}(-\ad(z_0)^2Z,Z).
\]
But, $\ad(z_0)|_{\got{p}}^2 = -\id_{\got{p}}$. Thus $\langle\Phi_{\Omega_{\got{p}}}(Z),z_0\rangle = \|Z\|^2$.
\end{proof}

We will also need the next lemma, which is an analogous of the Poincar\'e Lemma for smooth families of differential forms.

\begin{lem}
\label{lem:Poincarelemma_smoothfamilyversion}
Let $(\omega_t)_{t\in[0,1]}$ be a smooth family of closed $2$-forms on $\got{p}$. Then, there exists a smooth family $(\mu_t)_{t\in[0,1]}$ of $1$-forms, such that $\omega_t = d\mu_t$ for all $t\in[0,1]$. Moreover, if the $2$-form $\omega_t$ is $K$-invariant for $t\in[0,1]$, then we can take $\mu_t$ to be $K$-invariant.
\end{lem}

The proof of this lemma is almost the same as the one of the Poincar\'e Lemma for $2$-forms on $\R^n$ (see for example \cite[4.18]{warner}), by making obvious changes of notation. The parameter $t$ does not involve any change in the computations, and the result is actually a smooth family.

Moreover, if $\omega_t$ is $K$-invariant, then the linearity of the action of $K$ on $\got{p}$ induces that  $\mu_t$ is $K$-invariant. This can be checked directly on the definition of $\mu_t$ given in the proof of the Poincar\'e Lemma.

\begin{proof}[Proof of Theorem \ref{thm:mainthm_G/Kcase}]
For all $t\in[0,1]$, let $\Omega_t$ be the differential $2$-form on $\got{p}$ defined by
\[
\Omega_t|_Z := (\Gamma_0^*\Omega_{G\cdot\lambda_0})|_{tZ} \quad \forall Z\in\got{p}.
\]
In particular, $\Omega_1 = \Gamma_0^*\Omega_{G\cdot\lambda_0}$. Moreover, for $t=0$, we have the constant symplectic form
\[
\Omega_0|_Z(A,B) = \langle\lambda_0,[A,B]\rangle = B_{\theta}(z_0,[A,B]) = \Omega_{\got{p}}|_Z(A,B),
\]
for all $Z,A,B\in\got{p}$, since $\Psi_0^+ = \id_{\got{p}}$. When $t\neq 0$, one can check that
\[
\Omega_t = \frac{1}{t^2}\eta_t^*(\Gamma_0^*\Omega_{G\cdot\lambda_0}) = \frac{1}{t^2}\eta_t^*\Omega_1,
\]
where $\eta_t:\got{p}\rightarrow\got{p}$ is the homothecy $Z\mapsto tZ$, for all $t\in[0,1]$. By linearity of the action of $K$, $\eta_t$ commutes with this action. Thus $\Omega_t$ is $K$-invariant. Furthermore, since $\Omega_1$ is closed, we have $d\Omega_t = \frac{1}{t^2}d(\eta_t^*\Omega_1)=\frac{1}{t^2}\eta_t^*d\Omega_1 = 0$. But $\Omega_1$ is symplectic, so the skew-symmetric bilinear form $(\Omega_t)|_Z=(\Omega_1)|_{tZ}$ is clearly nondegenerate. We conclude that $\Omega_t$ is symplectic for all $t\in[0,1]$.

The smooth family $(\Omega_t)_{t\in[0,1]}$ of symplectic forms induces the smooth family $(\frac{d}{dt}\Omega_t)_{t\in[0,1]}$ of $K$-invariant closed $2$-forms. Indeed, $\frac{d}{dt}\Omega_t$ is closed for all $t\in[0,1]$ since the exterior derivative $d$ and the differential operator $\frac{d}{dt}$ commute. Now, from  Lemma \ref{lem:Poincarelemma_smoothfamilyversion}, there exists a smooth family $(\mu_t)_{t\in[0,1]}$ of $K$-invariant $1$-forms on $\got{p}$ such that, for all $t\in[0,1]$, we have $\frac{d}{dt}\Omega_t = d\mu_t$. This proves hypothesis 1) of Theorem \ref{thm:moserthm}.

For all $t\in]0,1]$, we define
\[
\Phi_t := \frac{1}{t^2}\eta_t^*\Phi_{\Gamma_0^*\Omega_{G\cdot\lambda_0}} - \frac{1}{t^2}\lambda_0 :\got{p}\rightarrow\got{k}^*,
\]
and, for $t=0$, we set $\Phi_0 := \Phi_{\Omega_{\got{p}}}$. The maps $\Phi_t$ are moment maps for the Hamiltonian $K$-manifolds $(\got{p}, \Omega_t)$, since $\lambda_0$ is centralized by $K$.

Note that, for all $t\in[0,1]$, we have $\Phi_t(0) = 0$. The set $\{\Phi_t(0); t\in[0,1]\}$ is thus reduced to a single point, and assertion 2) of Theorem \ref{thm:moserthm} is obviously satisfied. Moreover, the submanifold $M\times\{0\}=K\cdot\lambda_0\times\{0\}$ being identified to $\{0\}\subset\got{p}$, we have $(T_0\got{p})^{\Omega_t|_0} = \{0\} = T_0\{0\}$, that is, condition 3) is also satisfied.

It remains to prove hypothesis 4) of Theorem \ref{thm:moserthm}. A first computation gives
\[
\langle\Phi_t(Z),z_0\rangle = \frac{1}{t^2}\langle\Phi_{\Gamma_0^*\Omega_{G\cdot\lambda_0}}(tZ)-\lambda_0,z_0\rangle \geqslant \frac{1}{2t^2}\|tZ\|^2 = \frac{1}{2}\|Z\|^2,
\]
for all $Z\in\got{p}$ and all $t\in]0,1]$, using Lemma \ref{lem:momentmap_Phi_Gamma0OmegaGlambda0_is_proper}. But, by Lemma \ref{lem:momentmap_Phi_Omega_gotp_is_proper}, we also have $\langle\Phi_0(Z),z_0\rangle\geq \frac{1}{2}\|Z\|^2$ for all $Z\in\got{p}$. So $\|\Phi_t(Z)\| = \sup_{X\in\got{k}\setminus\{0\}}\frac{\langle\Phi_t(Z),X\rangle}{\|X\|}\geq\frac{1}{2\|z_0\|}\|Z\|^2$, for all $Z\in\got{p}$ and all $t\in[0,1]$. Finally assertion 4) is proved and we conclude by applying Theorem \ref{thm:moserthm} and the fact that the condition ``$\mu_t|_0(0)=0$ for all $t\in[0,1]$'' is always verified on the vector space $\got{p}\simeq K\cdot\lambda_0\times\got{p}$.
\end{proof}

\section{Proof of Theorem \ref{thm:mainthm}}
\label{section:proof_mainthm}

In this last section, we prove Theorem \ref{thm:mainthm} for any $\lambda\in\Chol$. Now, the two $K$-equivariant diffeomorphisms $\Gamma:K\cdot\lambda\times\got{p}\rightarrow G\cdot\lambda$ and $\Gamma_0:\got{p}\rightarrow G\cdot\lambda_0$ are involved, so that we will exclusively work on the manifold $K\cdot\lambda\times\got{p}$. We will consider the following symplectic forms on $K\cdot\lambda\times\got{p}$ ,
\begin{enumerate}
\item[(\emph{i})] $\Omega_{K\cdot\lambda\times\got{p}} = \Omega_{K\cdot\lambda}\oplus\Omega_{\got{p}}$;
\item[(\emph{ii})] $\Omega^1:= \Omega_{K\cdot\lambda}\oplus\Gamma_0^*\Omega_{G\cdot\lambda_0}$;
\item[(\emph{iii})] $\Omega^{\delta}:= \Omega_{K\cdot\lambda}\oplus (\delta\Gamma_0^*\Omega_{G\cdot\lambda_0})$, for all $\delta>0$;
\item[(\emph{iv})] $\Gamma^*\Omega_{G\cdot\lambda}$.
\end{enumerate}
Recall that the ``direct sum'' of two symplectic forms is defined as the canonical symplectic form on the direct product of the two underlying symplectic manifolds.

The purpose of this section is to prove that the symplectic forms $\Omega_{K\cdot\lambda\times\got{p}}$ and $\Gamma^*\Omega_{G\cdot\lambda}$ are symplectomorphic. To this end, we will use repeatedly the Moser argument given in section \ref{section:moser} in order to prove the existence of the symplectomorphisms indicated in the following diagram,
\begin{diagram}
\Omega_{K\cdot\lambda\times\got{p}} & \rTo^{\text{Theorem \ref{thm:mainthm_G/Kcase} }} & \Omega^1 & \rTo^{\text{Example \ref{exmp:vector_space_example_for_moserthm} }} & \Omega^{\delta} & \rTo^{\text{paragraph \ref{subsection:last_symplecto} }} \Gamma^*\Omega_{G\cdot\lambda}
\end{diagram}
The first symplectomorphism directly results from Theorem \ref{thm:mainthm_G/Kcase}, the second one from Example \ref{exmp:vector_space_example_for_moserthm} and  Lemma \ref{lem:momentmap_Phi_Gamma0OmegaGlambda0_is_proper}. The last arrow will be studied in the next paragraphs. Furthermore, one can assume that the diffeomorphisms given by the three arrows in the above diagram, fix each point of the submanifold $K\cdot\lambda\times\{0\}$. Composing such diffeomorphisms yields a symplectomorphism from $(K\cdot\lambda\times\got{p},\Omega_{K\cdot\lambda\times\got{p}})$ onto $(K\cdot\lambda\times\got{p},\Gamma^*\Omega_{G\cdot\lambda})$, which satisfies the statement of Theorem \ref{thm:mainthm}.

\subsection{Symplectomorphism between $\Omega^{\delta}$ and $\Gamma^*\Omega_{G\cdot\lambda}$ on $K\cdot\lambda\times\got{p}$}
\label{subsection:last_symplecto}

Then, the proof of Theorem \ref{thm:mainthm} will be completed by proving the next statement.

\begin{thm}
\label{thm:symplecto_Gamma*OmegaGlamnbda_and_Omegadelta}
For all $\delta> b_{\lambda}:=\sup_{\|u\|=1,\|v\|=1}\langle\lambda,[u,v]\rangle$, there exists a $K$-equivariant symplectomorphism from $(K\cdot\lambda\times\got{p},\Gamma^*\Omega_{G\cdot\lambda})$ onto $(K\cdot\lambda\times\got{p},\Omega^{\delta})$, which fixes each point $(k\lambda,0)$, for all $k\in K$.
\end{thm}

We will apply Corollary \ref{cor:moserthm_segment_and_poincarecondition} again. The main difficulty lies in proving that every $2$-form of the segment connecting the symplectic forms $\Omega^{\delta}$ and $\Gamma^*\Omega_{G\cdot\lambda}$, is symplectic too. According to the statement of the next theorem, this is possible for $\delta$ large enough.

\begin{thm}
\label{thm:segment_is_symplectic}
If $\delta>b_{\lambda}:=\sup_{\|u\|=1,\|v\|=1}\langle\lambda,[u,v]\rangle$, then, for all $t\in[0,1]$, the $2$-form $\Omega_t^{\delta} := t\Omega^{\delta}+(1-t)\Gamma^*\Omega_{G\cdot\lambda}$ is symplectic.
\end{thm}

This result will be proved in paragraph \ref{subsection:proof_of_thm_segment_is_symplectic}. We first need a lemma, which will be useful in the proofs of Theorems \ref{thm:symplecto_Gamma*OmegaGlamnbda_and_Omegadelta} and \ref{thm:segment_is_symplectic}.

We begin by setting some notations. 
For all $\lambda\in\got{t}^*$, let $H_{\lambda}$ be the unique element of $\got{t}$ such that
\[
B_{\theta}(H_{\lambda},X) = \langle\lambda,X\rangle \quad \forall X\in\got{g}.
\]
For any noncompact positive root $\beta$, we fix two nonzero vectors $E_{\beta}\in\got{g}_{\beta}$ and $E_{-\beta}\in\got{g}_{-\beta}$ such that $E_{-\beta} = \overline{E_{\beta}}$. Then, $E_{\beta}+E_{-\beta}$ and $i(E_{\beta}-E_{-\beta})$ are in $\got{p}$ (that is, they are real vectors). Moreover, the family $\bigl(E_{\beta}+E_{-\beta}, i(E_{\beta}-E_{-\beta})\bigr)_{\beta\in\got{R}_n^+}$ is a $\R$-basis of $\got{p}$, and it is well-known that this basis of $\got{p}$ is orthogonal for the inner product $B_{\theta}$. Moreover, we can choose $E_{\beta}$ and $E_{-\beta}$ such that
\begin{equation}
\label{eq:nomr_of_basisvectors_Ebeta}
B_{\theta}(E_{\beta}+E_{-\beta},E_{\beta}+E_{-\beta}) =  B_{\theta}\bigl(i(E_{\beta}-E_{-\beta}),i(E_{\beta}-E_{-\beta})\bigr)= 2,
\end{equation}
see \cite{knapp,helgason,bordemann}.


\begin{lem}
\label{lem:positivity_powers_ad(Z)_forHlambda}
Let $\lambda,\lambda'$ be in $\Chol$
. Then, for all $Z\in\got{p}$, we have
\begin{equation}
\label{eq:minoration_Btheta(HadZ2H')}
B_{\theta}(H_{\lambda},\ad(Z)^2 H_{\lambda'}) \geq \Bigl(\min_{\beta\in\got{R}_n^+}\beta(H_{\lambda})\beta(H_{\lambda'})\Bigr)\|Z\|^2.
\end{equation}
In particular, if we set $m_{\lambda}=\min_{\beta\in\got{R}_n^+}\beta(H_{\lambda})$, then
\begin{equation}
\label{eq:minoration_Btheta(z0adZ2H)}
B_{\theta}(z_0,\ad(Z)^2 H_{\lambda}) \geq m_{\lambda}\|Z\|^2,
\end{equation}
and
\begin{equation}
\label{eq:minoration_Btheta(HadZ2H)}
B_{\theta}(H_{\lambda},\ad(Z)^2 H_{\lambda}) \geq m_{\lambda}^2\|Z\|^2,
\end{equation}
for all $Z\in\got{p}$.
\end{lem}

\begin{proof}
Let $Z$ be in $\got{p}$. Then, 
\[
Z = \sum_{\beta\in\got{R}_n^+}\left(x_{\beta}^+(E_{\beta}+E_{-\beta})+x_{\beta}^-i(E_{\beta}-E_{-\beta})\right),
\]
with $x_{\beta}^{\pm}\in\R$ for all $\beta\in\got{R}_n^+$. Note that \eqref{eq:nomr_of_basisvectors_Ebeta} implies
\begin{equation}
\label{eq:Btheta(Z,Z)}
B_{\theta}(Z,Z) = \sum_{\beta\in\got{R}_n^+}2 \left((x_{\beta}^-)^2+ (x_{\beta}^+)^2\right).
\end{equation}

Let $H\in\got{t}$. Since $E_{\pm\beta}$ is in $\got{g}_{\pm\beta}$, we have $[H,E_{\pm\beta}] = \pm i\beta(H)E_{\pm\beta}$. Thus, we deduce the two equalities
\[
[H,E_{\beta} + E_{-\beta}] = \beta(H)\bigl(i(E_{\beta}-E_{-\beta})\bigr)
\]
and
\[
[H,i(E_{\beta}-E_{-\beta})] = -\beta(H)(E_{\beta} + E_{-\beta}).
\]
Consequently,
\[
[H,Z] = \sum_{\beta\in\got{R}_n^+}\beta(H)\left(-x_{\beta}^-(E_{\beta}+E_{-\beta})+x_{\beta}^+i(E_{\beta}-E_{-\beta})\right),
\]
for all $H\in\got{t}$.

Now, let $\lambda,\lambda'\in\Chol$, and denote by $H_{\lambda}, H_{\lambda'}\in\got{t}$ the respective dual elements of $\lambda$ and $\lambda'$ defined by the inner product $B_{\theta}$. These elements of $\got{t}$ necessarily verify $\beta(H_{\lambda})>0$ and $\beta(H_{\lambda'})>0$ for all $\beta\in\got{R}_n^+$. We have
\begin{align*}
B_{\theta}(H_{\lambda},\ad(Z)^2 H_{\lambda'}) & = B_{\theta}([H_{\lambda},Z],[H_{\lambda'},Z]) \\
& = \sum_{\beta\in\got{R}_n^+}2\beta(H_{\lambda})\beta(H_{\lambda'})\left((x_{\beta}^-)^2+ (x_{\beta}^+)^2\right).
\end{align*}
But $\beta(H_{\lambda})\beta(H_{\lambda'})$ is positive for all $\beta\in\got{R}_n^+$, thus one can obtain
\[
B_{\theta}(H_{\lambda},\ad(Z)^2 H_{\lambda'})\geq \Bigl(\min_{\beta\in\got{R}_n^+}\beta(H_{\lambda})\beta(H_{\lambda'})\Bigr)B_{\theta}(Z,Z),
\]
by \eqref{eq:Btheta(Z,Z)}. This proves equation \eqref{eq:minoration_Btheta(HadZ2H')}.

If we take $\lambda' = \lambda_0$, then $H_{\lambda'} = H_{\lambda_0} = z_0$. Since $\beta(z_0) = 1$ for all $\beta\in\got{R}_n^+$, we must have $\min_{\beta\in\got{R}_n^+}\beta(H_{\lambda})\beta(z_0) = m_{\lambda}$, and equation \eqref{eq:minoration_Btheta(z0adZ2H)} is clear. And finally, equation \eqref{eq:minoration_Btheta(HadZ2H)} is induced by the equality $\min_{\beta\in\got{R}_n^+}\left(\beta(H_{\lambda})^2\right) = m_{\lambda}^2$, which is true because of the positivity of the numbers $\beta(H_{\lambda})$.
\end{proof}

\begin{proof}[Proof of Theorem \ref{thm:symplecto_Gamma*OmegaGlamnbda_and_Omegadelta}]
We want to apply Corollary \ref{cor:moserthm_segment_and_poincarecondition}, so we first have to check hypotheses a) to d) of that statement.

By Theorem \ref{thm:segment_is_symplectic}, the condition $\delta> b_{\lambda}$ implies that all the elements of the family $(\Omega_t^{\delta})_{t\in[0,1]}$ are symplectic forms on $K\cdot\lambda\times\got{p}$. Hence, assertion a) is satisfied.

Secondly, from formula \eqref{eq:formula_GammaOmegaGlambda} and the definition of $\Omega^{\delta}$, one can easily show that the $2$-form $\Gamma^*\Omega_{G\cdot\lambda} - \Omega^{\delta}$ is in the kernel of $i^*$. Then, hypothesis 2) is verified.

Now consider the expression of $\Omega_t^{\delta}|_{(k\lambda,0)}$, for any $k\in K$. Since $\Psi_0(A) = \Psi_0^+(A) = A$ for all $A\in\got{p}$, we have
\[
\Omega_t^{\delta}|_{(k\lambda,0)}\bigl(([k,X],A),([k,Y],B)\bigr) = \langle\lambda,[X,Y]\rangle + \langle t\delta\lambda_0+(1-t)\lambda,[A,B]\rangle
\]
for all $X,Y\in\got{k}/\got{k}_{\lambda}$ and all $A,B\in\got{p}$. Thus, we clearly deduces assertion c) for $\Omega_t^{\delta}$.

It remains to check assertion d). A moment map for the Hamiltonian $K$-manifold $(K\cdot\lambda\times\got{p},\Gamma^*\Omega_{G\cdot\lambda})$ is the map defined for all $(k\lambda,Z)\in K\cdot\lambda\times\got{p}$ by
\[
\Phi_{\Gamma^*\Omega_{G\cdot\lambda}}(k\lambda,Z) := (e^Zk\lambda)|_{\got{k}} = k\lambda\circ\left(\sum_{n\geqslant 0}\frac{\ad(Z)^{2n}}{(2n)!}\right).
\]
Obviously, the map  $\Phi^{\delta}$ defined for all $(k\lambda,Z)\in K\cdot\lambda\times\got{p}$ by
\[
\Phi^{\delta}(k\lambda,Z) := k\lambda + \delta \lambda_0\circ\left(\sum_{n\geqslant 0}\frac{\ad(Z)^{2n}}{(2n)!}\right),
\]
is a moment map for $(K\cdot\lambda\times\got{p},\Omega^{\delta})$. Consequently, we obtain a moment map for $\Omega_t^{\delta}$ by taking $\phi_t^{\delta} := t\Phi_{\Gamma^*\Omega_{G\cdot\lambda}} + (1-t)\Phi^{\delta}$.

We define for any $t\in[0,1]$ the element $\lambda_t := t\lambda + (1-t)\delta \lambda_0$ of $\Chol$, and denote by $H_{\lambda_t} = tH_{\lambda} + (1-t)\delta z_0$ the associated vector in $\got{t}$.

Recall that $\Phi_t^{\delta}$ is $K$-equivariant, so we only need to consider the points $(\lambda,Z)$ with $Z\in\got{p}$. We make a first computation:
\begin{align*}
\langle\phi_t^{\delta}(\lambda,Z),H_{\lambda_t}\rangle & = \left\langle t\Phi_{\Gamma^*\Omega_{G\cdot\lambda}}(\lambda,Z)+(1-t)\Phi^{\delta}(\lambda,Z), H_{\lambda_t}\right\rangle\ , \\
& = t B_{\theta}\bigl(H_{\lambda},\sum_{n\geq 0}\frac{\ad(Z)^{2n}}{(2n)!}H_{\lambda_t}\bigr) \\
& \quad + (1-t)\delta B_{\theta}\bigl(z_0,\sum_{n\geq 0}\frac{\ad(Z)^{2n}}{(2n)!}H_{\lambda_t}\bigr) \\
& \quad + (1-t)\langle \lambda, H_{\lambda_t}\rangle\ , \\
& = B_{\theta}\bigl(H_{\lambda_t},\sum_{n\geq 0}\frac{\ad(Z)^{2n}}{(2n)!}H_{\lambda_t}\bigr) + (1-t)\langle\lambda, H_{\lambda_t}\rangle\ .
\end{align*}
Since $\lambda$ and $\lambda_t$ are both in $\Chol$, the number $\langle\lambda,H_{\lambda_t}\rangle$ is positive, because $\langle\lambda,H_{\lambda_t}\rangle = 2\sum_{\alpha\in\got{R}^+}\alpha(H_{\lambda})\alpha(H_{\lambda_t}) > 0$. But, $\ad(Z)$ is symmetric for $B_{\theta}$, then
\begin{align*}
\langle\phi_t^{\delta}(\lambda,Z),H_{\lambda_t}\rangle & \geq \sum_{n\geq 0}\frac{1}{(2n)!}B_{\theta}\bigl(\ad(Z)^{n}H_{\lambda_t},\ad(Z)^{n}H_{\lambda_t}\bigr) \\
& \geq \frac{1}{2}B_{\theta}\bigl(\ad(Z)H_{\lambda_t},\ad(Z)H_{\lambda_t}\bigr).
\end{align*}
Equation \eqref{eq:minoration_Btheta(HadZ2H)} of Lemma \ref{lem:positivity_powers_ad(Z)_forHlambda} yields now
\[
\langle\phi_t^{\delta}(\lambda,Z),H_{\lambda_t}\rangle \geqslant \frac{m_{\lambda_t}^2}{2}\|Z\|^2.
\]
Finally, we get the following inequalities, for all $(k\lambda,Z)\in K\cdot\lambda\times\got{p}$,
\[
\|\phi_t^{\delta}(k\lambda,Z)\| \geq \frac{m_{\lambda_t}^2}{2\|H_{\lambda_t}\|}\|Z\|^2 \geq \Bigl(\inf_{t\in[0,1]}\frac{m_{\lambda_t}^2}{2\|H_{\lambda_t}\|}\Bigr)\|Z\|^2,
\]
because the norms are $K$-invariant. Note that the constant number $\inf_{t\in[0,1]}\frac{m_{\lambda_t}^2}{2\|H_{\lambda_t}\|}$ is positive, since $H_{\lambda_t}$ is never zero and by continuity on the compact set $[0,1]$. Therefore, assertion d) of Corollary \ref{cor:moserthm_segment_and_poincarecondition} is checked.
\end{proof}

\subsection{Proof of Theorem \ref{thm:segment_is_symplectic}}
\label{subsection:proof_of_thm_segment_is_symplectic}

Define, for all $Z\in\got{p}$, the two nondegenerate skew-symmetric bilinear forms $\omega_0^{\delta}|_Z$ and $\omega_1|_Z$ on $\got{k}/\got{k}_{\lambda}\oplus\got{p}$, by
\[
\omega_0^{\delta}|_Z\bigl((X,A),(Y,B)\bigr) = \langle\lambda,[X,Y]\rangle + \delta B_{\theta}(z_0,[A,B]),
\]
for any $\delta >0$, and
\[
\omega_1|_Z\bigl((X,A),(Y,B)\bigr) = \langle\lambda,[X+\chi_Z(A),Y+\chi_Z(B)]\rangle + \langle\lambda,[A,B]\rangle,
\]
for all $(X,A), (Y,B) \in \got{k}/\got{k}_{\lambda}\oplus\got{p}$.

\begin{lem}
\label{lem:necessarycondition_if_linearforms_opposed}
Let $\delta>0$. Assume that there exists $Z\in\got{p}$, $(X,A)\in\got{k}/\got{k}_{\lambda}\oplus\got{p}$ nonzero and $c>0$ such that
\begin{equation}
\label{eq:linearforms_opposed}
\omega_1|_Z\bigl((X,A),(Y,B)\bigr) = -c\,\omega_0^{\delta}|_Z\bigl((X,A),(Y,B)\bigr),
\end{equation}
for all $(Y,B)\in\got{k}/\got{k}_{\lambda}\oplus\got{p}$. Then $\delta \leq b_{\lambda}$.
\end{lem}

\begin{proof}
Let $Z\in\got{p}$, $(X,A)\in\got{k}/\got{k}_{\lambda}\oplus\got{p}$ nonzero and $c>0$, such that \eqref{eq:linearforms_opposed} is valid for all $(Y,B)\in\got{k}/\got{k}_{\lambda}\oplus\got{p}$. Then, taking $B=0$, we have
\[
\langle\lambda,[X+\chi_Z(A),Y]\rangle = -c\langle\lambda,[X,Y]\rangle, \quad \forall Y\in\got{k}/\got{k}_{\lambda}.
\]
This yields that $X+\chi_Z(A) = -c X \mod \got{k}_{\lambda}$, that is, $\chi_Z(A) = -(c+1)X \mod \got{k}_{\lambda}$. But, by linearity of $\chi_Z$, we cannot have $A=0$, because $X$ would also be zero, which would contradict the hypothesis ``$(X,A)\neq 0$''. The equation \eqref{eq:linearforms_opposed} is now reduced to
\[
\langle\lambda,[X+\chi_Z(A),\chi_Z(B)]\rangle+\langle\lambda,[A,B]\rangle = -c\delta B_{\theta}(z_0,[A,B]),
\]
satisfied for all $B\in\got{p}$. It is equivalent to
\begin{equation}
\label{eq:equality_necessarycondition_1}
\delta B_{\theta}(z_0,[A,B]) + \frac{1}{c}\langle\lambda,[A,B]\rangle = \langle\lambda,[X,\chi_Z(B)]\rangle,
\end{equation}
for all $B\in\got{p}$, since $X+\chi_Z(A) = -c X \mod \got{k}_{\lambda}$. Taking $B=-[z_0,A]\in\got{p}$ in \eqref{eq:equality_necessarycondition_1}, we get
\begin{align*}
0 & < \delta B_{\theta}(z_0,[A,-[z_0,A]]) + \frac{1}{c}\langle\lambda,[A,-[z_0,A]]\rangle \\
& \leq -\frac{1}{c+1}\langle\lambda,[\chi_Z(A),\chi_Z(-[z_0,A])]\rangle
\end{align*}
Note that $\|A\|^2 = B_{\theta}(A,A) = B_{\theta}(z_0,[A,-[z_0,A]]))$. Moreover, we have $\frac{1}{c}\langle\lambda,[A,-[z_0,A]]\rangle = \frac{1}{c}B_{\theta}(z_0,\ad(A)^2 H_{\lambda})\geq 0$, from \eqref{eq:minoration_Btheta(z0adZ2H)} of Lemma \ref{lem:positivity_powers_ad(Z)_forHlambda} applied to $A\in\got{p}$. We deduce the following inequalities,
\begin{align*}
\delta\|A\|^2 & \leq \delta B_{\theta}(z_0,[A,-[z_0,A]]) + \frac{1}{c}\langle\lambda,[A,-[z_0,A]]\rangle \\
& \leq \frac{1}{c+1}\langle\lambda,[\chi_Z(A),\chi_Z([z_0,A])]\rangle \\
& \leq \frac{1}{c+1} b_{\lambda}\|\chi_Z(A)\|.\|\chi_Z([z_0,A])\|.
\end{align*}
The number $c$ is positive, so $0<\frac{1}{c+1}\leq 1$. Furthermore, by Lemma \ref{lem:chi_Z_and_eigenvalues}, the linear operator $\chi_Z$ is symmetric for $B_{\theta}$ and all its eigenvalues are in $]-1,1[$. Hence,
\[
\|\chi_Z(W)\| \leq \|W\| \quad \forall W\in\got{g}.
\]
Consequently, $\|\chi_Z([z_0,A])\| \leq \|[z_0,A]\| = \|A\|$. So $\delta\|A\|^2 \leq b_{\lambda}\|A\|^2$, with $A\neq 0$. Then we conclude that $\delta\leq b_{\lambda}$.
\end{proof}

\begin{proof}[Proof of Theorem \ref{thm:segment_is_symplectic}]
We are going to prove that $\Omega_t$ is non-degenerate at any point of $K\cdot\lambda\times\got{p}$, and leave the other verifications to the reader. Since $\Omega_t$ is $K$-invariant, it is enough to show that it is nondegenerate at the points of $\{\lambda\}\times\got{p}$.

Let $\delta>b_{\lambda}$. By contraposition of the statement of Lemma  \ref{lem:necessarycondition_if_linearforms_opposed}, for all $Z\in\got{p}$, all $(X,A)\in\got{k}/\got{k}_{\lambda}\oplus\got{p}$  and all $c>0$, we must have the following inequality:
\[
\imath\bigl((X,A)\bigr)\omega_1|_Z \not\equiv -c\, \imath\bigl((X,A)\bigr)\omega_0^{\delta}|_Z.
\]
It means that, for all $Z\in\got{p}$ and $X\oplus A\in\got{k}/\got{k}_{\lambda}\oplus\got{p}\setminus\{(0,0)\}$, there exists a vector $(Y,B)\in\got{k}/\got{k}_{\lambda}\oplus\got{p}$ such that
\[
\langle\lambda,[X,Y]\rangle+\delta B_{\theta}(z_0,[A,B]) >0
\]
and
\[
\langle\lambda,[X+\chi_Z(A),Y+\chi_Z(B)]\rangle+\langle\lambda,[A,B]\rangle >0.
\]
Indeed, it is a basic fact from linear algebra that, if two nonzero linear forms  $f_1$ and $f_2$ on a finite dimensional $\R$-vector space $F$ satisfy the assertion
\[
\forall x\in F, \quad f_1(x) >0 \quad \Longrightarrow \quad f_2(x)\leq 0,
\]
then they have the same kernel, and thus there exists a constant number $c\in \R^*$ such that $f_2 = cf_1$. But clearly one must have $c<0$, otherwise $f_1(x)>0$ implies $f_2(x)=cf_1(x)>0$.

Now, from \eqref{eq:defn_chiZ} and \eqref{eq:formula_GammaOmegaGlambda} and the definition of $\Omega^{\delta}$, the two above inequalities ensure that
\[
\Omega^{\delta}|_{(\lambda,Z)}\left(\left([1,X],(\Psi_Z^+)^{-1}(A)\right),\left([1,Y],(\Psi_Z^+)^{-1}(B)\right)\right) >0
\]
and
\[
(\Gamma^*\Omega_{G\cdot\lambda})|_{(\lambda,Z)}\left(\left([1,X],(\Psi_Z^+)^{-1}(A)\right),\left([1,Y],(\Psi_Z^+)^{-1}(B)\right)\right) > 0.
\]
So we clearly have
\[
\Omega_t^{\delta}|_{(\lambda,Z)}\left(\left([1,\overline{X}],(\Psi_Z^+)^{-1}(A)\right),\left([1,\overline{Y}],(\Psi_Z^+)^{-1}(B)\right)\right) > 0
\]
for all $t\in[0,1]$, which proves that $\Omega_t^{\delta}|_{(\lambda,Z)}$ is nondegenerate.
\end{proof}

\bibliographystyle{plain}
\bibliography{bibliography}

\end{document}